\documentclass{amsart}
\usepackage{amsmath,amsfonts,amssymb,stmaryrd}
\usepackage[all]{xy}
\usepackage{graphicx,psfrag,subfigure}

\newtheorem{maintheorem}{Theorem}

\newtheorem{thm}{Theorem}[section]
\newtheorem{rk}{Remark}
\newtheorem{prop}{Proposition}[section]

\newtheorem{lema}{Lemma}[subsection]

\newcommand{\R}{{\mathbb{R}}}

\newcommand{\Z}{{\mathbb{Z}}}
\newcommand{\T}{{\mathbb{T}}}
\newcommand{\N}{{\mathbb{N}}}
\newcommand{\B}{ $\rule{1.2mm}{2mm}$\\}

\begin{document}


\title[Robust Transtivity for Endomorphisms admitting critical points]{ROBUST TRANSITIVITY FOR ENDOMORPHISMS ADMITTING CRITICAL POINTS}

\author[J. Iglesias]{J. Iglesias}
\address{Universidad de La Rep\'ublica. Facultad de Ingenieria. IMERL.
Julio Herrera y Reissig 565. C.P. 11300. Montevideo, Uruguay}
\email{jorgei@fing.edu.uy }

\author[C. Lizana]{C. Lizana}
\address{Departamento de Matem\'aticas. Facultad de Ciencias.
Universidad de los Andes. La Hechicera-M\'erida 5101.
Venezuela} \email{clizana@ula.ve}

\author[A. Portela]{A. Portela}
\address{Universidad de La Rep\'ublica. Facultad de Ingenieria. IMERL.
Julio Herrera y Reissig 565. C.P. 11300. Montevideo, Uruguay }
\email{aldo@fing.edu.uy }

\date{\today}

\maketitle
\begin{abstract}
We address the problem of giving necessary and sufficient conditions in order to have
robustly transitive endomorphisms admitting persistent critical sets. 
We exhibit different type of open examples of robustly transitive maps
 in any isotopic class of endomorphisms acting on the two dimensional torus
admitting persistent critical points.
We also provide some necessary condition for robust transitivity in this setting.
\end{abstract}

\section{Introduction and Statement of the Main Results}

An important goal in dynamics is to study conditions that preserve main  pro\-per\-ties of a given system under perturbations.
In particular, we are interested in robust transitivity, understanding 
by transitivity that  there exists a point with dense forward orbit and by robust transitivity
that all nearby systems of a transitive one are also transitive.
In the diffeomorphisms setting the behavior of such a systems is quite well understood.
It was proved in  \cite{bdp} for compact manifolds of any dimension
that robust transitivity implies a weak form of hyperbolicity, so-called dominated splitting.
 In dimension two, R. Ma\~{n}\'e proved in (\cite{m2}) that robustly transitive implies hyperbolicity
 and moreover, that the only surface that admit such systems is the 2-torus.
 The first examples of non-hyperbolic robustly transitive diffeomorphisms
 were given by  M. Shub (\cite{sh}) in $\T^4$ and by R. Ma\~{n}\'e (\cite{m1}) in $\T^3$.
 In the endomorphisms (non-invertible maps) setting, this kind of behavior is not longer true.
 For endomorphisms, hyperbolicity is not a necessary condition
 in order to have robust transitivity.
The first work that address the problem about necessary and sufficient conditions
for robustly transitive endomorphisms was \cite{lp} and later some examples appeared in \cite{hg}.
In \cite{lp} it was proved that volume expanding is a $C^1-$necessary condition for robustly
 transitive local diffeomorphisms not exhibiting
a dominated splitting (in a robust way). However, it is not a sufficient condition. For further details see \cite{lp}.
\footnote{AMS classification:  37D20, 37D30, 08A35, 35B38. Key words: robustly transitive endomorphisms, critical set.}

So far the study of robust transitivity has been done in the local diffeomorphisms setting, meanwhile
the case of endomorphisms admitting critical points has received less
far attention. The unique known example before this work of a  robustly transitive endomorphism
 with critical points was constructed in \cite{br}. 
We address in this paper the problem of giving necessary and sufficient conditions for having robustly transitive
endomorphisms admitting persistent critical set.
Since we are admitting the existence of critical points, we have that our examples of robustly transitive maps
are not volume expanding, appearing the main difference between the local diffeomorphisms setting
and ours. Hence, volume expanding is not a necessary condition for having robustly transitive endomorphisms
in general. The construction of robustly transitive local diffeomorphisms examples that are
not hyperbolic varies according its homotopy class. In \cite{lp} shows examples in the expanding homotopy class and
\cite{hg} shows examples in the non-hyperbolic and hyperbolic homotopy class. These constructions differ substantially
between them.
The technics used in the diffeomorphisms and local diffeomorphisms setting are not easily adapted under the
presence of critical points. We inspired ourselves in those works to construct our examples.




Let us fix some notation before we state the main results of this work. 
Let $f$ be an {\bf endomorphism} in $\T^2$ that is $f$ is a non-invertible map and
$f_{*}$ the map induced by $f$ in the fundamental group of $\T^2$.
The map $f_{*}$ can be represented by a square matrix of size two by two with integer coefficients.
Since we are interested in the expanding volume setting that is
the module of the determinant of these matrices are greater or equal to two,
then it cannot have non-real eigenvalues of module one.
Consider $\mathcal{M}^{*}_{2\times 2}(\Z) \subset \mathcal{M}_{2\times 2}(\Z) $
 the subset of matrices with two different real eigenvalues and  module of determinant greater or equal to two.
 We denote by $S_f$ the set of critical points of $f,$ that is the set of points
where $Df$ is non-invertible, calling it by {\bf critical set}.
We say that the critical set is {\bf persistent} if there exists $C^1-$neighborhood of $f$ such that all maps
in this neighborhood has non-empty critical set. Hence, the setting is endomorphisms  on the 2-torus
 admitting critical sets.

 In this context, we prove that in the homotopy class of a given matrix $A\in\mathcal{M}^{*}_{2\times 2}(\Z),$
 always there exists a
 $C^1-$robustly transitive map admitting a family of unstable cones and
 a non-empty persistent critical set. Let us state our main result.

\begin{maintheorem}
\label{principal}
 \emph{ For every matrix $A\in \mathcal{M}^{*}_{2\times 2}(\Z),$  there exist $f$ homotopic to
  $A$ and $\mathcal{U}_f$ $C^{1}$-neighborhood   of $f$ such that for
  all $g\in\mathcal{U}_f$ holds that $g$ is transitive, admits a family of unstable cones  and $S_g$ is non-empty.}
\end{maintheorem}

Some remarks about our main result are in order.
As we mentioned above the issue of having robust transitivity plus persistent critical sets
has never been studied before and our result cover a large
class of examples of robustly transitive maps with persistent
critical set. We can restated Theorem \ref{principal} as follows, for e\-very $C^0-$neighborhood of a matrix
$A\in \mathcal{M}^{*}_{2\times 2}(\Z),$ always there exists $C^1-$robustly transitive endomorphism
with persistent critical set admitting an ``expanding direction''.
Meaning by ``expanding direction'' the existence of a family of unstable cones that
are transversal to the kernel of the differential $Df$ at the critical set. We use strongly
in the proof of Theorem \ref{principal} the existence of such a family. Roughly speaking the idea of the proof is the following.
Since we wish to prove robust
transitivity for some map in the homotopy class of $A,$ given any open set, using the existence of the unstable
cones we know that for a finite number of forward iterates we reach a large enough diameter to get out from
the ``bad region", that is the region where we have the critical points. Once we get a large enough diameter
we obtain a point such that its forward orbit is expanding, hence the internal radius of the iterates of the initial open
set growth. Since the initial matrix has several points with dense pre-orbit, the perturbation has $\varepsilon-$dense pre-orbit,
then any pre-orbit
by the perturbation intersect the open set above, finishing the proof.

Since a 2 by 2 matrix with integer coefficients and determinant
greater or equal to two does not admit non-real eigenvalues
in the unit circle, 
we have basically three
options for these type of matrices, two eigenvalues of module greater than 1 (expanding case), one eigenvalue greater than 1 and the other
less than 1 (saddle case), and one eigenvalue equal 1 and the other greater than 1 (non-hyperbolic case).
Because of that we divided the proof of Theorem \ref{principal} in three parts, the expanding (Proposition \ref{prop1}),
the non-hyperbolic (Proposition \ref{prop11})
and
 the saddle (Proposition \ref{prop111})
cases. We splitted the proof in three cases because each homotopy class is approached in a different way.
Unfortunately we do not know a technic that allow us to prove our result without considering the three different
construction.

Some question that arise from the above discussion are the following.
{\it What happens in other surfaces or in torus of higher dimension}?
{\it Is the existence of that type of cones fields  a necessary condition for
having robustly transitive maps admitting critical points}?
{\it Could be the dimension of the kernel
of $Df$ equal to the dimension of the whole manifold}?
We expect that we can generalized Theorem \ref{principal} for
torus in higher dimension but we are not quite sure if it
can be done in other surfaces because of Ma\~{n}\'{e} results.
Concerning the existence of
the unstable cones fields transversal to the kernel of $Df$, we expect
that it is a necessary condition in order to have robust transitivity under the assumption of existence
of critical points.

In this direction, we first prove the following proposition that play a key role in the proof of Theorem \ref{teoA}.
This follows from the existence of critical
points and because we are considering the $C^1-$topology.

\begin{prop}\label{prop-teoA}
 Given $f\!\in \! C^1(M)$ with persistent critical set, 
there exists $g$ $C^1-$close to $f$ with non-empty interior of the critical set.
\end{prop}

In our next result, we prove that if $M$ is a two dimensional compact manifold, the dimension of
$\mathrm{ker}(Df)$ less or equal $1$ is a necessary condition for
robust transitivity. Concretely, 

\begin{maintheorem}
\label{teoA}
For every $f\in C^1(M)$  robustly transitive map with persistent
critical set follows that:
\begin{enumerate}
\item[1)] There exist $g$ $C^1-$close to $f$ and $x_0$ a critical point for $g$ such that the
forward orbit of $x_0$ by $g$ is dense.
\item[2)] $Df(x,y)$ is  different from the null matrix for every $(x,y)\in S_f.$
\end{enumerate}
\end{maintheorem}

Since  Theorem (Whitney) \ref{teo1}  holds for two dimensional manifolds and  we use it for proving our result, we
assume that the dimension of the manifold is two.
This is another reason why we prove Theorem \ref{principal} in dimension two. It seems reasonable to expect similar results
for higher dimension.

The paper is organized as follows. In Section~\ref{sec3} we present
the proof of Theorem \ref{principal}
splitting  in three cases as we mentioned above,
expanding (subsection \ref{expC}), non-hyperbolic (subsection \ref{nhc})
and saddle case (subsection \ref{sc}),
presenting a large class of robustly transitive endomorphisms admitting persistent critical sets. In
Section~\ref{sec2} some definitions are given and we prove Proposition \ref{prop-teoA} and Theorem  \ref{teoA}.


\section{Proof of Theorem \ref{principal}}\label{sec3}

\subsection{Expanding case}\label{expC}

In this  section we  consider the case when both eigenvalues of $A$ are greater than $1$. Concretely,

\begin{prop}\label{prop1}
Given a matrix $A\in \mathcal{M}^{*}_{2\times 2}(\Z)$  with eigenvalues
$\lambda$ and $\mu$ with $|\lambda |>|\mu |>1,$ there
exist $f$ homotopic to $A$ and a $C^{1}$-neighborhood  $\mathcal{U}_f$ of $f$ such that
for all $g\in\mathcal{U}_f$ holds that $g$ is transitive, admits a family of unstable cones  and $S_g$ is non-empty.
  \end{prop}

\subsubsection{\bf{Sketch of the Proof}}
Let $A$ as in Proposition \ref{prop1}. Since  $|\lambda |>|\mu |,$ there exist
 unstable cones fields in the direction of $\lambda.$ Denote by $f$ the perturbation 
 of $A$ in a convenient small
 ball centered at the origin such that $f$ coincide with $A$ in the complement of the ball
 and there still exist a family of unstable cones (Lemma \ref{conosinestables}). 
 Then for any given open set, there exist a point such that its forward orbit by $f$ 
 remain in the expanding region (Lemma \ref{lema44}). Hence, in a finite number of 
 forward iterates, the internal radius growth until reaching some fix radius (Lemma \ref{lema3}). Finally,
$A$ has several points with dense pre-orbit, so $f$ has $\varepsilon-$dense pre-orbit,
then any pre-orbit
by $f$ intersect the open set above finishing the proof (Lemma \ref{lema66} and Lemma \ref{lemma7}). 

\subsubsection{\bf{Construction of $f$}}
Let  $A$ be a matrix with  spectrum $\sigma (A)=\{\lambda, \mu \}$,
$\lambda ,\mu \in \R$ and  $|\lambda |>|\mu |>1$.
After a change of  coordinates, if necessary, we may assume that
$$A=\left(\begin{array}{cc}
  \lambda & 0  \\
  0 & \mu  \\
\end{array}%
\right).$$

Let us consider
\begin{itemize}
 \item $\delta_A:=sup\{r>0:  \ A|_{B(x,r)} \mbox{ is a diffeomorphism onto its image } \}$,
 \item $U:=B(0,4r)$ and  $U':=B(0,3r)$ with $4r<\delta_A$.
 \end{itemize}

    Since the eigenvalues of $A$ are greater than one, it follows that
     $A(U\setminus U^{'})\cap U'=\emptyset$.  Let $r$ be small enough such that
   \begin{equation}\label{p1}
 A^{i}(U\setminus U')\cap U'=\emptyset, \mbox{ for }
i=1,...,9.
\end{equation}

Consider $U''=B(0,r/8).$
Fix $\theta >0$ such that $\overline{B(0, 2\theta )}\subset U''$ and consider
$\psi :\R\to\R$  $C^{\infty}-$class such that $\psi (x)=\psi (-x)$,
$x=0$ the unique critical point, $\psi (0)=2\mu$ and $\psi (x)=0$ for
 $|x|\geq\theta$ (see Figure \ref{figura1}(a)). Let $\varphi:\R\to\R$ be such
 that $\varphi (y)=0$ for $y\notin [0,\delta ]$ and  $\varphi' $ is as in
 Figure \ref{figura1}(b), with $\delta <\theta$. Now defined
 $f_{\theta ,\delta }:\R^{2}\to \R^{2}$ by
 $$f_{\theta ,\delta }(x,y)=(\lambda x,\mu y-\psi (x)\varphi (y)).$$

\begin{figure}[ht]
\psfrag{1}{\tiny{$1$}}\psfrag{dd}{\tiny{$\delta/2$}}
\psfrag{d}{\tiny{$\delta$}}
\psfrag{c}{\tiny{$2\mu$}}
\psfrag{a}{\tiny{$-\theta$}}
\psfrag{aa}{\tiny{$\theta$}}
\psfrag{q}{$\psi$}
\psfrag{p}{$\varphi'$}
\begin{center}
\subfigure[]{\includegraphics[scale=0.15]{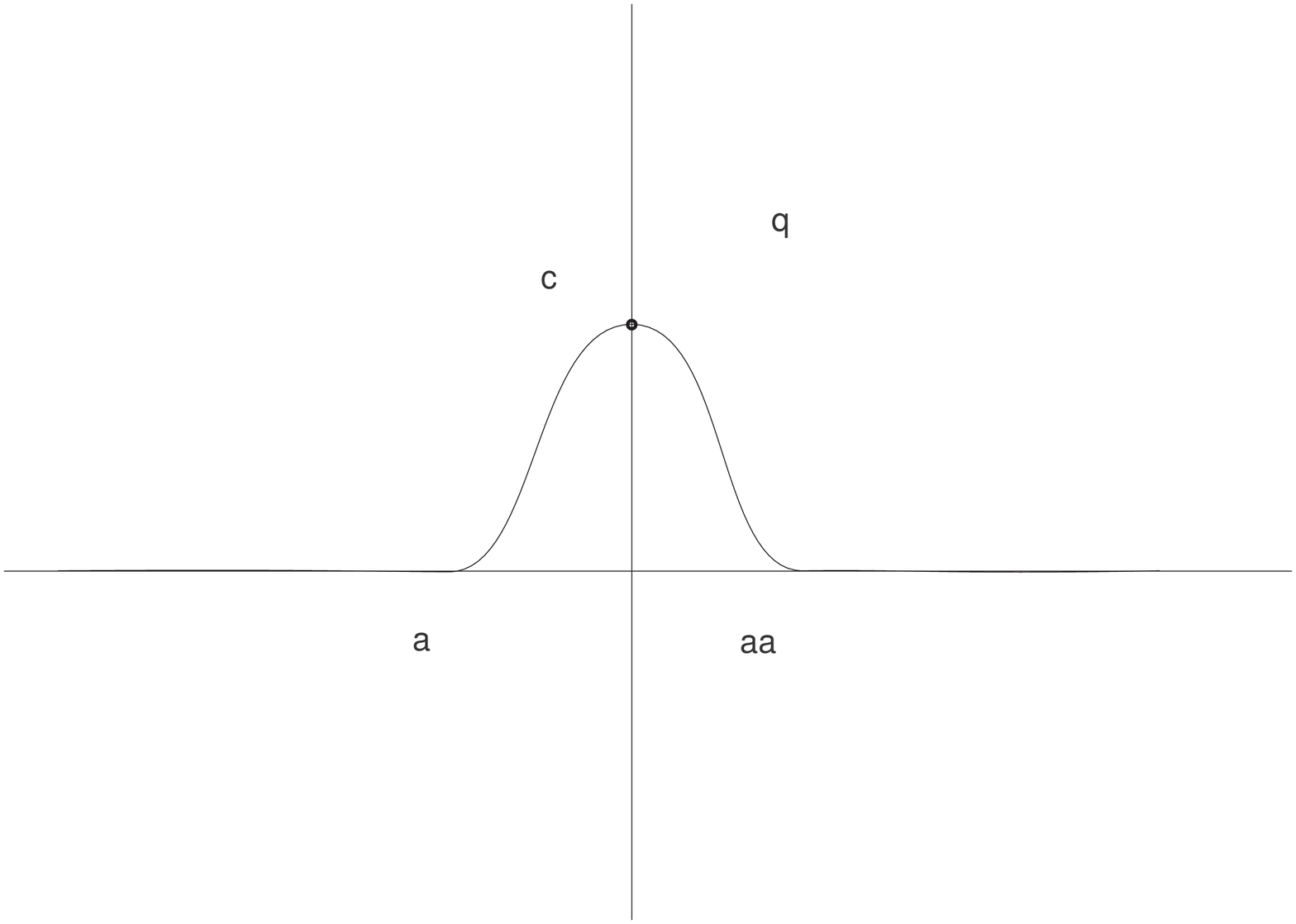}}
\subfigure[]{\includegraphics[scale=0.15]{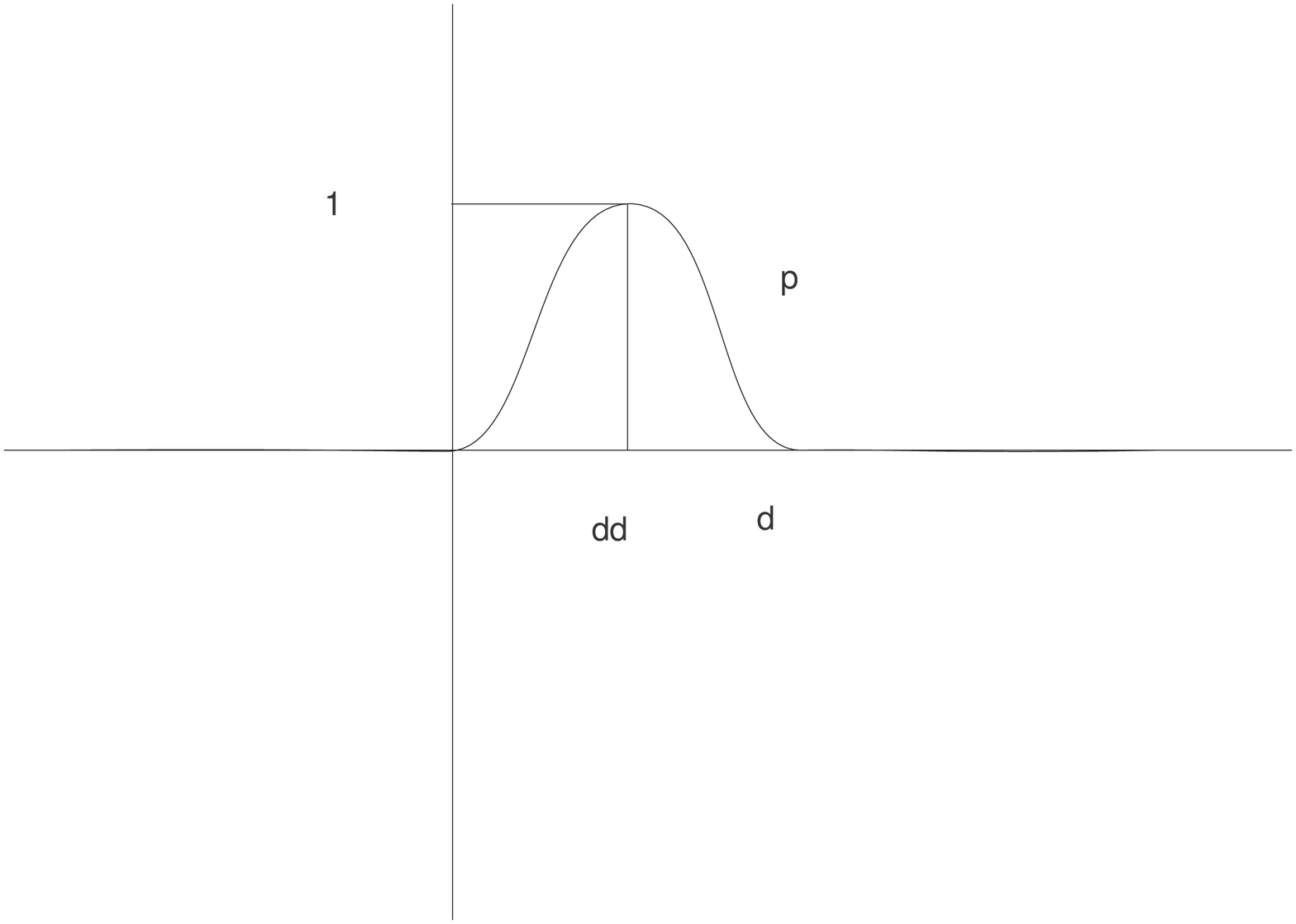}}
\caption{Graph of $\psi$ and $\varphi'$}\label{figura1}
\end{center}
\end{figure}

 For simplicity, we denote $f_{\theta ,\delta }=f$, although $f$ depends on the parameters $\theta$ and $\delta$.

\begin{rk}\label{rk1}
   The following properties are very useful for our purpose. Since they are not hard to prove
   we leave it for the reader to verify.
   \begin{enumerate}
    \item[a)] $f\mid_{\T^{2}\setminus U''} =A_{\T^{2}\setminus U''}.$
    \item[b)] From \eqref{p1} it follows that $f^{i}(U\setminus U')\cap U'=\emptyset,$ for $i=1,...,9.$
    \item[c)] The critical set of $f$ is $S_f=\{(x,y): \mu-\psi (x)\varphi' (y)=0 \}$.
    \item[d)]  For $x=0$ and $y=\delta /2$ we have that $det (Df_{(0,\delta /2)})=-\lambda \mu $
        and for $(x,y)\in \T^{2}\setminus U$ we have that  $det (Df_{(x,y)})=\lambda  \mu$.
        Then there exists a $C^{1}$-neighborhood  $\mathcal{U}_f$ of $f$ such that
        $S_g\neq\emptyset$ for all $g\in\mathcal{U}_f$.
    \item[e)] $f$ goes to $A$ in the $C^{0}$ topology, when  $\theta$ and $\delta$ go to zero.
    \item[f)]  $f$ is homotopic to $A$.
    \end{enumerate}
\end{rk}

Given $a\in \R$ positive and  $p\in\T^{2}$, we
consider $\mathcal{C}^u_a(p)\subset T_p(\T^{2})$    \emph{the family of unstable cones} defined by
$\mathcal{C}_a^u (p)=\{(v_1,v_2)\in T_p(\T^{2}) : \  |v_2|/|v_1| < a \}$.
The following lemma shows that it is possible to construct a family of unstable cones
for the map $f$. For the statement of the following lemma we use the fact that $|\lambda |>\sqrt{2},$
this follows from $|det(A)|\geq 2$.

\begin{lema}[Existence of unstable cones for $f$]\label{conosinestables}
Given $\theta >0$, $a>0$, $\delta >0$  and $\lambda'$ with $\sqrt{2}\leq \lambda'<|\lambda |$, there exist $a_0>0$
and $\delta_0 >0$ with $0<a_0<a$ and $0<\delta_0 <\delta$ such that if
$f=f_{\theta ,\delta_{0} },$ then the following properties hold:
\begin{enumerate}
\item[$(i)$] If $\mathcal{C}_{a_{0}}^u(p)\!\!=\!\!\{(v_1,v_2):   |v_2|/|v_1| <a_0 \},$
        then $\overline{Df_p(\mathcal{C}_{a_{0}}^u(p))}\setminus \{(0,0)\}\subset \mathcal{C}_{{a_{0}}}^u(f(p))$,
        for all  $p\in \T^{2}$;
\item[$(ii)$] If $v\in \mathcal{C}_{a_{0}}^u(p),$ then $|Df_p(v)|\geq \lambda'|v|;$ and
\item[$(iii)$] If $\gamma$ is a curve such that $\gamma'(t)\subset \mathcal{C}^u_{a_{0}}(\gamma(t)),$
         then $\mathrm{diam}(f(\gamma))\geq \lambda' \mathrm{diam}(\gamma )$.
\end{enumerate}
\end{lema}

\begin{proof}
Proof of $(i)$. Given $p=(x,y)\in\T^2$ and $a>0,$ pick $a_0$ such that $0<a_0<a.$  Let $v=(v_1,v_2)\in \mathcal{C}_{a_0}^u(p)$. Since

 $$\begin{array}{ll}
                         Df_{(x,y)}(v_1,v_2)&= \left(
                                                 \begin{array}{cc}
                                                 \lambda & 0  \\
                                                 -\psi'(x)\varphi (y) & \mu-\psi (x)\varphi' (y)  \\
                                                 \end{array}
                                               \right)
                                                 \left(\begin{array}{c}
                                                     v_1  \\
                                                      v_2  \\
                                                  \end{array}%
                                                   \right) \\ \\
                          & = ( \lambda v_1, -\psi'(x)\varphi(y)v_1+ (\mu-\psi (x)\varphi' (y))v_2),
                       \end{array}
                       $$
then

 \begin{equation}\label{eq1}
 \frac{|-\psi'(x)\varphi (y)v_1+ (\mu-\psi (x)\varphi' (y))v_2 |}{|\lambda v_1|}\! \leq\!
 \left|\frac{-\psi'(x)\varphi (y)}{\lambda}\right|\! +\! \left|\frac{\mu-\psi (x)\varphi' (y)}{\lambda}\right|  \left|\frac{v_2}{v_1}\right|.
\end{equation}

 Let $M=\max\{| \psi'|\}$. Note that $\max\{| \varphi|\}\leq \delta$
 and $|\mu-\psi (x)\varphi' (y)|\leq |\mu|.$ Hence, from inequality (\ref{eq1}) follows that
$$     \frac{|-\psi^{'}(x)\varphi (y)v_1+ (\mu-\psi (x)\varphi' (y))v_2 |}{|\lambda v_1|} \leq
\frac{M\delta}{|\lambda|}+\frac{|\mu|}{|\lambda|} \left|\frac{v_2}{v_1}\right|   
<  \frac{M\delta}{|\lambda|}+\frac{|\mu |a_0}{|\lambda|}.$$
Since $|\lambda| >|\mu|$, taking $\delta_0$ small enough we obtain that
$\frac{M\delta_0}{|\lambda|}+\frac{|\mu |a_0}{|\lambda|}<a_0$ which finishes the proof of (i).

\noindent Proof of $(ii)$. Let $\lambda'$ be such that $\sqrt{2}\leq\lambda'<|\lambda|$. Note that
$$\begin{array}{ll}
    \left (\dfrac{|Df_{(x,y)}(v_1,v_2)|}{|\lambda'(v_1,v_2)|} \right )^{2} &= \dfrac{ (\lambda v_1)^{2} +(-\psi'(x)\varphi (y)v_1+ (\mu-\psi (x)\varphi' (y))v_2)^{2}}{ (\lambda')^{2}(v_1^{2}+v_2^{2}) } \\ \\
     & = \dfrac{ \lambda^{2} +(-\psi'(x)\varphi (y)+ (\mu-\psi (x)\varphi' (y))\frac{v_2}{v_1})^{2}}{ (\lambda')^{2}\left(1+\left (\frac{v_2}{v_1}\right)^{2}\right) }.
  \end{array}
 $$

%
%

Taking $\delta_0$  and $a_0$ small enough, we get that $-\psi'(x)\varphi (y)$ is arbitrarily close to zero,
 $(\mu-\psi (x)\varphi' (y))\frac{v_2}{v_1}$ is also close to
zero  and $\left(1+\left(\frac{v_2}{v_1}\right)^{2}\right)$ is close to one.
Then 
there exist $a_0$ and $\delta_0$, as close to zero as necessary, such that
$$  \left (\frac{|Df_{(x,y)}(v_1,v_2)|}{|\lambda^{'}(v_1,v_2)|} \right )^{2}>1,$$
and the thesis follows.

\noindent Proof of $(iii)$. It follows from the previous items.
\end{proof}

Note that 
the properties $(i)$, $(ii)$ and $(iii)$ in the previous lemma are robust.
In concrete we have the following result.

\begin{lema}\label{clly1}
For every $f$ satisfying 
 Lemma \ref{conosinestables}, there exists $\mathcal{U}_f$ a $C^{1}$-neighborhood
 of $f$ such that for every $g\in\mathcal{U}_f$
 the properties $(i)$, $(ii)$ and $(iii)$ of Lemma \ref{conosinestables} hold.
\end{lema}

\begin{proof}
 We leave the proof as an easy exercise for the reader to verify that
 $(i)$, $(ii)$ and $(iii)$ of Lemma \ref{conosinestables} are open properties.
\end{proof}

\begin{lema}\label{lema44}
If $f$ is as in Lemma \ref{conosinestables}, then there exists $\mathcal{U}_f$ a $C^{1}$-neighborhood
of $f$ such that for every $g\in\mathcal{U}_f$ and every $V\subset \T^{2}$ open set,
 there exist $y\in V$ and $n_0\in \N$ such that   $g^{n}(y)\in \T^{2}\setminus U'$ for all $n\geq n_0$.
\end{lema}

\begin{proof}

 Let $\mathcal{U}_f$ be a $C^{1}$-neighborhood of $f$ such that for all $g\in \mathcal{U}_f$ the following properties hold:
  \begin{itemize}
  \item $g$ satisfies the thesis of Lemma \ref{conosinestables}, taking $\lambda^{'}=\sqrt{2}$,
  \item $g^{i}(U\setminus U')\cap U'=\emptyset$, for
$i=1,...,9$ (see the definition of $f$) and
\item $g\mid_{\T^{2}\setminus U'}$ is expanding, i.e. the eigenvalues
of $Dg(p)$ are greater than one for $p\in \T^{2}\setminus U'$
(this is possible because $f\mid_{\T^{2}\setminus U''}=A\mid_{\T^{2}\setminus U''}$).
\end{itemize}

\noindent{\bf{Claim 1:}} If $\alpha$ is a curve such that  $\alpha'(t)\subset \mathcal{C}^u_{a_0}(\alpha(t))$,
$\alpha\subset \T^{2}\setminus U'$ and
  $diam(\alpha)\geq 10r$, then there exist $m\in\N$ and curves
  $\alpha_0,...,\alpha_m$ with $\alpha_0=\alpha$ such that:

  \begin{enumerate}
 \item $\alpha_i\subset \T^{2}\setminus U'$, $\forall i=0,...,m$,
  \item $g(\alpha_i)\supset \alpha_{i+1}$,  $\forall i=0,...,m-1$ and
  \item $diam(\alpha_m)\geq 10r$.
\end{enumerate}

Note that since $\alpha\subset \T^{2}\setminus U',$ then $diam (g(\alpha ))\geq 10r$.

If $g(\alpha)\cap U'=\emptyset$, then taking $m=1$ and $\alpha_1=g(\alpha )$ we are done.
If $g(\alpha)\cap U'\neq \emptyset$,
since $diam (U')=6r$ and  $diam (U)=8r$, there exists
$\alpha_1\subset g (\alpha )$ with $diam (\alpha_1)\geq r$ and  $\alpha_1\subset U\setminus U'$.
So we define $\alpha_{i+1}=g^{i}(\alpha_1)$ for $i=1,...,9$. By hypothesis
 $g^{i}(U\setminus U')\cap U'=\emptyset$, for
$i=1,...,9$, therefore $\alpha_i\subset \T^{2}\setminus U'$ and $diam
(\alpha_i)\geq (\sqrt{2})^{i}r$ for $i=1,...,9$. Since $(\sqrt{2})^{9}\geq
10$, then $diam(\alpha_{10})\geq 10r$. In this case, it is enough
to take $m=10$.

\noindent {\bf{Claim 2:}} Let $V$ be a open set and let $\gamma$ be a curve contained in
 $V$ with $\gamma'(t)\subset \mathcal{C}^u_{a_{0}}(\gamma(t))$.
There exist a sequence of curves $\{\gamma_n\}$ and $n_0\in\N$ such that:
\begin{enumerate}
  \item $\gamma_0\subset g^{n_{0}}(\gamma$), $g(\gamma_n)\supset \gamma_{n+1}$ and
  \item $\gamma_n\subset \T^{2}\setminus U'$ for every $n\geq n_0$.
\end{enumerate}

 By Lemma \ref{conosinestables} there exists $n_0\in \N$ such that $g^{n_{0}}(\gamma)$
 contains a curve $\gamma_0$ with $diam (\gamma_0)\geq 10r$ and $\gamma_0\subset \T^{2}\setminus U'$.
 Taking  $\alpha =\gamma_0$ and using claim 1 we obtain the sequence $\{\gamma _n\}$.

Given an open set $V$ and a curve $\gamma$ contained in $V$ with
$\gamma'(t)\subset \mathcal{C}^u_{a{_{0}}}(\gamma(t))$, let $\{\gamma_n\}$
be the sequence given by claim 2.
For every $\gamma_n,$ let $\gamma'_n\subset\gamma$ be such that
$g^{n_{0}+n}(\gamma'_n)=\gamma_n$. Since
$g(\gamma_n)\supset\gamma_{n+1}$ then $\gamma'_{n+1}\subset
\gamma'_n$. Consider $B=\cap\gamma'_n$. If $y\in B,$ then
 $g^{n}(y)\in \T^{2}\setminus U'$ for all $n\geq n_0$, and this proves the lemma.
\end{proof}

From now on we  assume that $f$ satisfies the thesis of
 Lemma \ref{conosinestables}.


\begin{lema}\label{lema3}
 There exists $\mathcal{U}_f$ $C^{1}$-neighborhood  of $f$ such that if
 $g\in\mathcal{U}_f$ and $y\in\T^2$ satisfies that $g^{n}(y)\in \T^{2}\setminus U'$
 for all $n\in \N$, then if $V$ is an open neighborhood of $y$, there exists $n_{_{V}}\in\N$
 such that if $n\geq n_{_{V}}$, $g^{n}(V) \supset B(g^{n}(y),r)$.
\end{lema}

\begin{proof}
 It follows from the fact that
$f\!\mid_{\T^{2}\setminus U''}= A\!\mid_{\T^{2}\setminus U''}$ and $dist(\partial U'', \partial U')>2r$.
\end{proof}

\begin{lema}\label{lema66}
 Given $\varepsilon >0,$ there exist open sets $B_1,...,B_n$ and $m\in\N$ such that:
 \begin{enumerate}
  \item[(i)] $\cup_{i=1}^{n}B_i=\T^{2}$, $diam (B_i)<\varepsilon $ and $A^{m}(B_i)=\T^{2}$ for every $i=1\ldots,n$;
  \item[(ii)] There exists a $C^{0}$-neighborhood $\mathcal{U}_A$ of $A$ such that
	      $g^{m}(B_i)=\T^{2}$  for all $g\in  \mathcal{U}_A,$ for every $i=1\ldots,n$.
 \end{enumerate}
\end{lema}

\begin{proof}
 Proof of (i). Take a finite covering of $\T^{2}$ by disks of radius less than $\varepsilon/4$.
 Let $B_1,...,B_n$ be such disks. As $A$ is expanding, there exists $m\in \N$ such that
 $A^{m}(B_i)=\T^{2}$, for $i=1,...,n$.

\noindent Proof of (ii). We begin by proving that for a fixed $i$ there exists $\delta_i>0$ such that
 for all $x\in\T^{2}$ there exists $V\subset B_i$ such that $A^{m}:V\to B(x,\delta_i)$ is an
 homeomorphism. In fact, let us assume  that there exists a  sequence  $\{ x_k \}$ with $x_k\to x$
 such that $A^{m}:V\to B(x_k,1/k)$  is not an homeomorphism for any $V$ with $V\subset B_i$.
 Because of  $A^{m}(B_i)=\T^{2},$ there exists
$y\in B_i$ such that $A^{m}(y)=x$, and there exists $r>0$ such that $A^{m}:B(y, r)\to A^{m}(B(y,r))$
is an homeomorphism. Let $k$ be a large enough positive integer such that
$B(x_k,1/k)\subset A^{m}(B(y,r))$ and $V_k=A^{-m}(B(x_k,1/k))\cap B(y,r).$
Clearly $A^{m}:V_k\to B(x_k,1/k)$ is an homeomorphism and this is a contradiction.

Let $\delta =\min\{\delta_i: \mbox{ for } i=1,...,n\}$ and
$\mathcal{U}_A$ be a $C^{0}$-neighborhood of $A$ such that
$$\mathcal{U}_A=\{g: d(g^m,A^m)<\delta/8 \}.$$

For every $g\in\mathcal{U}_A$ holds that $ g^m(B_i)=\T^{2},$ for $i=1,...,n$. In fact,
given $x\in \T^{2}$,  using the argument above we know that for every $i=1,...,n$
there exists an open set $V_i\subset B_i$ such that $A^{m}:V_i\to B(x,\delta)$
is an homeomorphism.
Since $d(g^m,A^m)<\delta/8,$ it follows that $ g^m(V_i) \supset B(x,\delta /16)$ for every $i=1,...,n$.
Therefore $ g^m(B_i)=\T^{2},$ since $x$ is arbitrarily.  
\end{proof}

The following lemma will be very useful for proving Proposition \ref{prop1}.
Since it is not hard to verify we omit its proof, for further details see \cite{lp}
and \cite{lpv}.

\begin{lema}\label{lemma7}
   Let $g:\T^{2}\to\T^{2}$ be such that the pre-orbit
 $\{w\in g^{-n}(x): n\in \N\}$ is dense in $\T^{2}$ for all $x\in\T^{2}$, then
 $g$ is transitive.
\end{lema}

\begin{proof}[{\textbf{Proof of Proposition \ref{prop1}.}}]
For $\varepsilon =r/4$,  let $\mathcal{U}_A$ be a $C^{0}$-neighborhood of $A$ such that Lemma
\ref{lema66} holds. As $f_{\theta , \delta}$ converges to $A$ in the
$C^{0}$-topology  when $\theta$ and $\delta$ goes to zero (see  Remark \ref{rk1}[item (e)]),
we may choose $\theta_0$ and $\delta_0$ such that $f_{\theta_0 , \delta_0}\in \mathcal{U}_A$ and
Lemma \ref{conosinestables} holds. Consider $\mathcal{U}_1$ a $C^{1}$-neighborhood of $f_{\theta_0 , \delta_0}$ with
$\mathcal{U}_1\subset \mathcal{U}_A$ such that Lemmas \ref{lema44} and \ref{lema3}, and item (d) of Remark \ref{rk1} hold.

Let us prove now that for all $g\in \mathcal{U}_1$, $S_g\neq\emptyset$ and $g$ is transitive.
Using Remark \ref{rk1}[item (d)] we get that  $S_g\neq\emptyset$.
By Lemma \ref{lemma7}, it is enough to prove that $\{w\in g^{-n}(x): n\in \N\}$ is dense in $\T^{2}$,
for all $x\in\T^{2}$. Given an open set
$V\subset\T^{2}$, Lemma \ref{lema44} implies there exist $y\in V$ and
$n_{0}\in\N $ such that $g^{n}(y)\in \T^{2}\setminus U'$ for all $n\geq n_0$. By Lemma \ref{lema3}
there exists $n_1$ such that $g^{n}(V)\supset B(g^{n}(y),
r)$ for all $n\geq n_1$ . Let  $B_1,...,B_n$ and $m\in \N$ be 
given by Lemma \ref{lema66} and $n_2=\max\{n_0,n_1\}$. As $diam (B_i)<r/4$, there exists $i_0$ such
that $g^{n_{2}}(V)\supset B_{i_{0}}$. By Lemma \ref{lema66}, item (ii), $g^{m}(B_{i_{0}})=\T^{2}$,
so given $x\in\T^{2}$ then  $x\in g^{m}(B_{i_{0}})$. Hence $x\in g^{n_2+m}(V)$ and therefore
$g^{-(n_2+m)}(x)\cap V\neq\emptyset$. Thus, the pre-orbit $\{w\in g^{-n}(x): n\in \N\}$ is dense in $\T^{2}$ for every point. 
\end{proof}



\subsection{Non-hyperbolic case}\label{nhc}
This section is devoted for the case when one of the eigenvalues is equal one.
Consider a matrix $A$ with spectrum $\sigma (A)=\{\lambda, 1 \}$, $\lambda \in \Z$ and $\lambda>5$.
We may assume that the matrix $A$ is  $\left(\begin{array}{cc}
 \lambda & 0  \\
0 & 1  \\
\end{array}%
\right)$ up to change of coordinates if necessary.
Here we prove our main result for the non-hyperbolic case.

\begin{prop}\label{prop11}
Given a matrix  $A\in \mathcal{M}^{*}_{2\times 2}(\Z)$  with eigenvalues
$|\lambda |>|\mu |=1,$ there exist $f$ homotopic  to $A$ and $\mathcal{U}_f$
$C^{1}$-neighborhood
 of $f$ such that for all $g\in\mathcal{U}_f$ holds that $g$ is transitive, admits a family of unstable cones  and $S_g$ is non-empty.
 \end{prop}

\subsubsection{\bf{Sketch of the Proof}}
Let $A$ as in Proposition \ref{prop11}. Inspired by Example 3.6 of \cite{hg},
we construct a function $f$ perturbing $A$ in such a way that appear two 
saddle fix points, two repelling fix points (Remark \ref{rk3}) and critical points.
We show that the stable and unstable manifold of one of the saddle point
are dense (Lemma \ref{lema8} and Lemma \ref{lema9}). Hence, $f$ is transitive,
non-hyperbolic, admits non-empty critical set and exhibits a family
of unstable cones. Finally, by construction follows that the latter properties are robust 
proving the main theorem.

\subsubsection{\bf{Construction of $f$}}
From now on we consider $\T^{2}=\R^{2}/[-1,1]^{2}$.
We begin constructing a function $f_0$ without critical points that will be
very helpful in the final construction of the function $f$ satisfying  Proposition \ref{prop11}.

Let  $h,g:S^{1}\to S^{1}$ be diffeomorphisms as in Figure \ref{figura6} (a)
with the following additional conditions:

\begin{enumerate}\label{enumerate1}
\item $x_0=0$ and $y_0$ are attracting fixed points for $g$ and $h$ respectively and
there exists $\lambda\in (0,1)$ such that  $|g'(x)|\leq \lambda $ for
$x\in [0,g^{-1}(h(0))]$ and $|h'(x)|\leq \lambda $ for $x\in [0,y_0]$ and
\item $x_1$ and $y_1$ are repelling fixed points for $g$ and $h$ respectively,
$|g'(x)|\leq \lambda/2 $ and $|h'(x)|\leq \lambda/2 $  for $x\in S^{1}$.
\end{enumerate}

\begin{figure}[ht]
\psfrag{a0}{\tiny{$A_0$}}\psfrag{a1}{\tiny{$A_1$}}
\psfrag{g}{\tiny{$g$}}
\psfrag{h}{\tiny{$h$}}\psfrag{p0}{\tiny{$x_0$}}\psfrag{p1}{\tiny{$x_1$}}
\psfrag{q0}{\tiny{$y_0$}}\psfrag{q1}{\tiny{$y_1$}}
\psfrag{h0}{\tiny{$h(0)$}}
\psfrag{p00}{\tiny{$p_0$}}
\psfrag{p11}{\tiny{$p_1$}}
\psfrag{q00}{\tiny{$q_0$}}
\psfrag{q11}{\tiny{$q_1$}}
\psfrag{gh}{\tiny{$g^{-1}(h(0))$}}
\begin{center}
\subfigure[]{\includegraphics[scale=0.15]{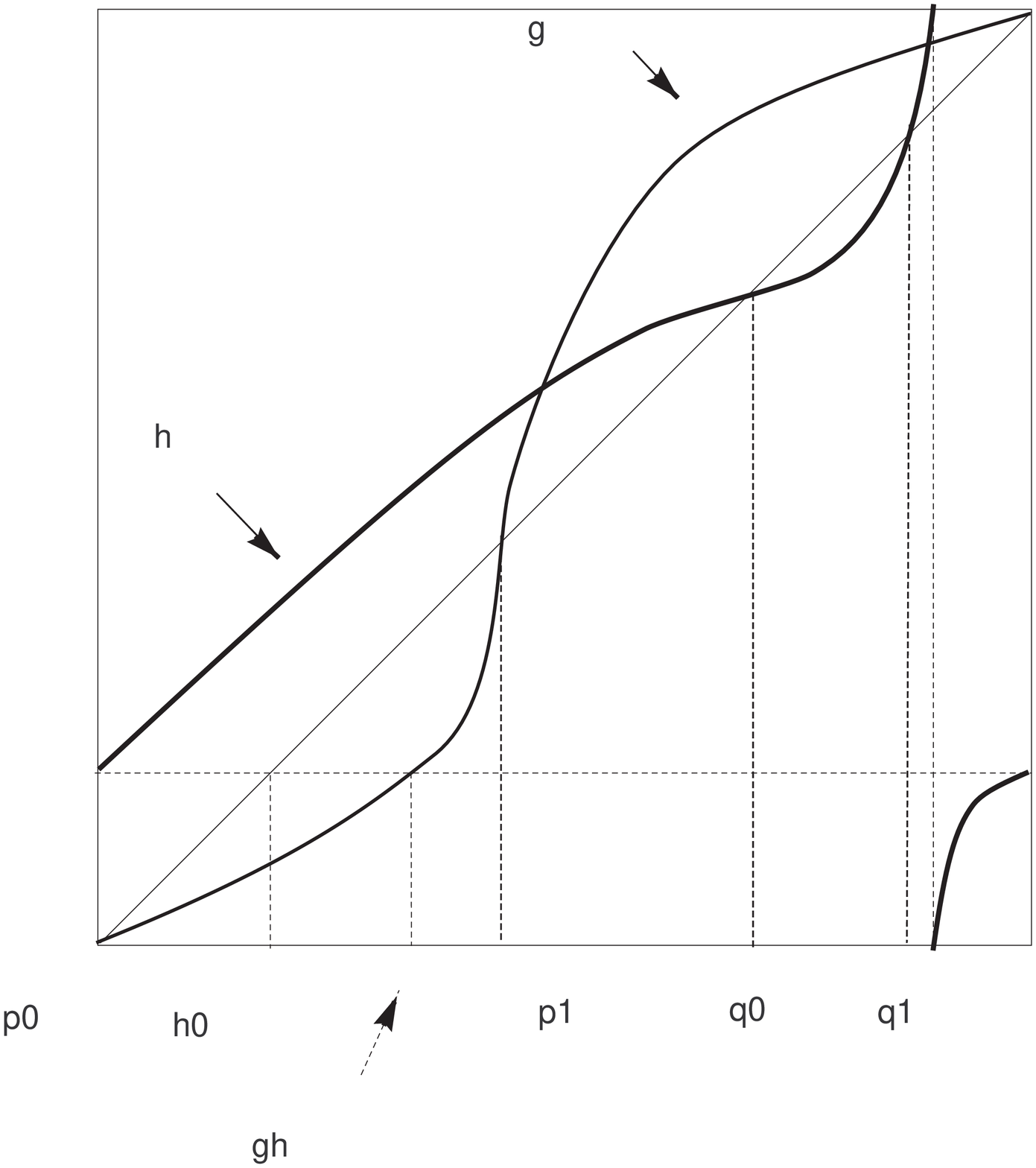}}
\subfigure[]{\includegraphics[scale=0.2]{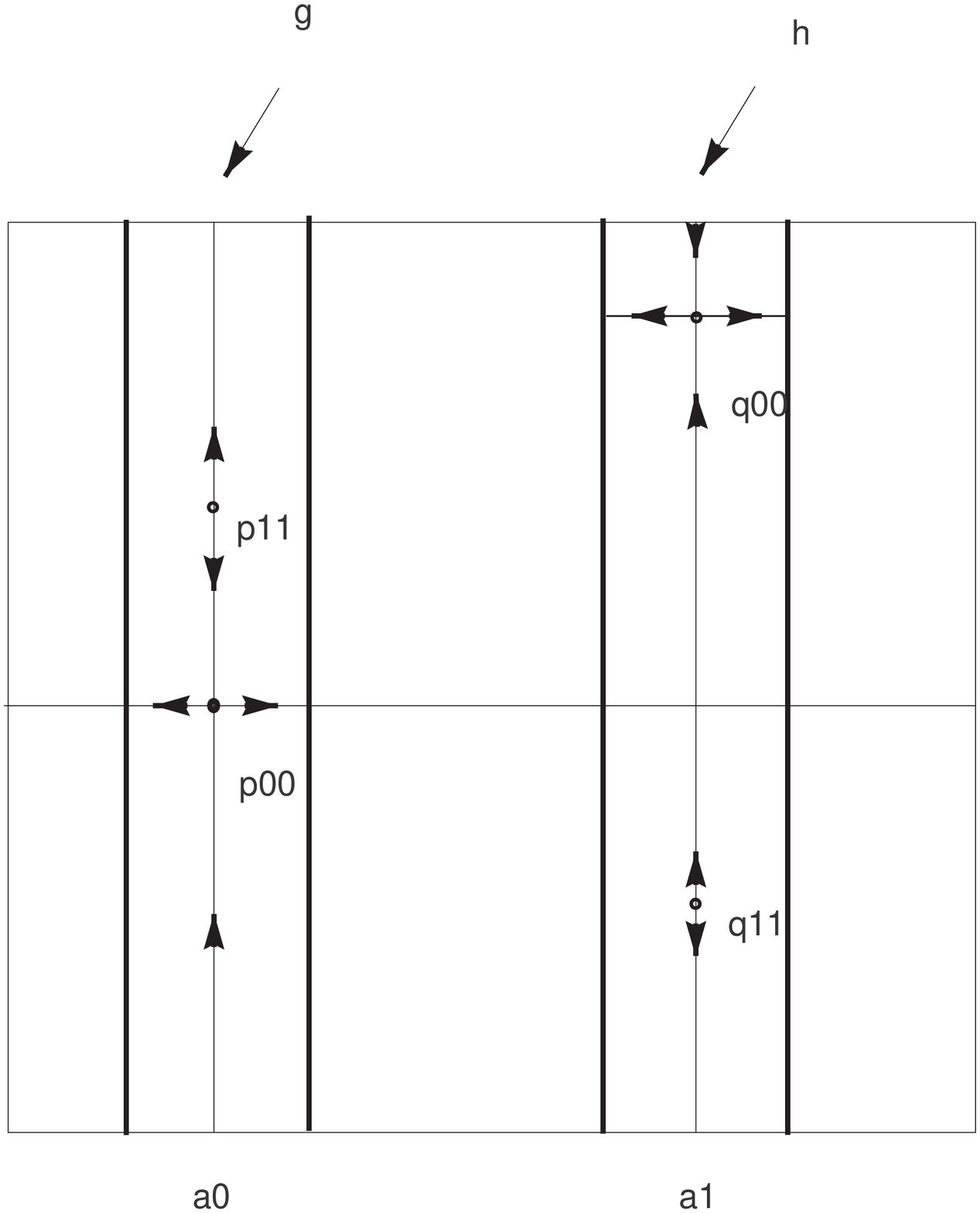}}
\caption{\label{figura6}}
\end{center}
\end{figure}


Consider  $\varepsilon =\frac{1}{100\lambda}$, $A_0=[-\frac{1}{2\lambda}-\varepsilon , \frac{1}{2\lambda}+\varepsilon] \times S^{1}$,
$A_1=[\frac{3}{2\lambda}-\varepsilon, \frac{5}{2\lambda}+\varepsilon]\times S^{1}$ and $f_0:A_0\cup A_1\to\T^{2}$ such that

$$f_{0}(x,y) =\left\{
\begin{array}{l}
(\lambda x,g(y)), \mbox{ if } (x,y)\in A_0,\\
{}\\
(\lambda x-2,h( y)), \mbox{ if } (x,y)\in A_1.
\end{array}
\right.$$

As $h$ and $g$ can be constructed $C^{1}$-close to the identity
as necessary, $f_0$ can be extended to $\T^{2}$ in such a way that in
the complement of a neighborhood of $A_0\cup A_1$, $f_0$ coincides with the matrix
$A$ and  $f_0$ has strong dominated expansion in the first coordinate. So there exists a family
$\mathcal{C}^u_a$ of unstable cones, where every curve  $\gamma$
such that $\gamma'(t)\subset \mathcal{C}^u_a(\gamma(t))$, $diam( f_0(\gamma))\geq \lambda'diam(\gamma)$ where $1<\lambda'<\lambda$.


\begin{rk}\label{rk3}
Some highlights about the dynamics of $f_0$.
\begin{enumerate}
\item Since $\lambda\!>\!5,$ it follows that $A_0\!\subset\! [-\frac{1}{2}\!-\varepsilon,\frac{1}{2}\!+\varepsilon]\!\times S^{1}$ and
 $A_1\!\subset\! [-\frac{1}{2}\!-\varepsilon,\frac{1}{2}\!+\varepsilon]\!\times S^{1}$.

 \item $f_0(A_0)\supset [-\frac{1}{2}-\varepsilon,\frac{1}{2}+\varepsilon]\times S^{1}$ and
 $f_0(A_1)\supset [-\frac{1}{2}-\varepsilon,\frac{1}{2}+\varepsilon]\times S^{1}$.

\item $f_0$ has two saddle fix points $p_0=(0,0)\in A_0$ and
$q_0=(\frac{2}{\lambda-1},y_0)\in A_1$ and two repelling fix points
$p_1=(0,x_1) \in A_0$ and $q_1=(\frac{2}{\lambda-1},y_1)\in A_1$.

\item If we denote by $W_{loc}^{u}(p_0)$ the set
$\{(s,0): -\frac{1}{2\lambda}-\varepsilon <s< \frac{1}{2\lambda}+\varepsilon \}$, then
$f_0(W_{loc}^{u}(p_0))\supset \{ (s',0):  -\frac{1}{2}-\varepsilon<s'< \frac{1}{2}+\varepsilon   \}.$

\item  $f(\{(s,0): -\frac{3}{2\lambda}-\varepsilon <s< \frac{5}{2\lambda}+\varepsilon
\})\supset \{ (s',h(0)):   -\frac{1}{2}-\varepsilon<s'< \frac{1}{2}+\varepsilon   \}.$

\item $f_0$ is homotopic to $A$.

\end{enumerate}
\end{rk}

\begin{lema}\label{lema8}
The unstable manifold of the point $p_0$ is dense in $\T^{2}$.
\end{lema}

\begin{proof}

Let us prove first that the unstable manifold of  $p_0$ is dense in
$[ -\frac{1}{2}-\varepsilon , \frac{1}{2}+\varepsilon ]\times [0,h(0)]$.  If it is not dense
then there exists an open set  $V\subset [ -\frac{1}{2}-\varepsilon, \frac{1}{2} +\varepsilon]\times [0,h(0)]$
such that  $V\cap W^{u}(p_0)=\emptyset$. Let  $s$ be a vertical segment contained in $V$.\\

\noindent {\bf{Claim}}: There exists a sequence of vertical segments  $\{s_n\}_{n\leq 0}$ such that
$f_0(s_{n-1})=s_n$, $s_0=s$ and
$s_n \subset (A_0\cup A_1)\cap [ -\frac{1}{2}-\varepsilon, \frac{1}{2} +\varepsilon]\times [0,g^{-1}(h(0))]$  for $n<0$.  

\vspace{1mm}
\noindent
\textit{Proof of the Claim}.
The idea of the proof is simple, taking as many pre-images  of the segment $s$ in $A_0$ or in $A_1$
as necessary to obtain the sequence $\{s_n\}_{n\leq 0}.$

Since $f_0(A_0)\supset [-\frac{1}{2}-\varepsilon,\frac{1}{2}+\varepsilon]\times S^{1},$ there
exists a vertical segment $s_{-1}\subset [-\frac{1}{2\lambda}-\varepsilon,\frac{1}{2\lambda}+\varepsilon]\times [0,g^{-1}(h(0))]\cap A_0$
such that  $f_0(s_{-1})=s$. Note that $s_{-1}\subset [-\frac{1}{2\lambda}-\varepsilon ,\frac{1}{2\lambda}+\varepsilon ]\times [0,h(0))$
or $s_{-1}\subset [-\frac{1}{2\lambda}-\varepsilon ,\frac{1}{2\lambda}+\varepsilon ]\times (h(0),g^{-1}(h(0)))$.
    If the first inclusion is verified, we take   $s_{-2}\subset A_0$ such that $f_0(s_{-2})=s_{-1}$.
If the second inclusion is verified, as  $f_0(A_1)\supset [-\frac{1}{2}-\varepsilon,\frac{1}{2}+\varepsilon]\times S^{1}$,
we take $s_{-2}\subset A_1$ such that $f_0(s_{-2})=s_{-1}$. In the case where
$s_{-2}\subset [-\frac{1}{2\lambda}-\varepsilon ,\frac{1}{2\lambda}+\varepsilon ]\times [0,h(0))$
we proceed as in the case where $s_{-1}\subset [-\frac{1}{2\lambda}-\varepsilon ,\frac{1}{2\lambda}+\varepsilon ]\times [0,h(0))$
and we define $s_{-3}$ as the pre-image by $f_0$ of $s_{-2}$ that is included in $A_0$.

In the case where we have either $s_{-2}\subset [-\frac{1}{2\lambda}-\varepsilon ,\frac{1}{2\lambda}+\varepsilon ]\times (h(0),g^{-1}(h(0)))$
or $s_{-2}\subset A_1$, define
$s_{-3},..., s_{-n_{0}}$ as the successive pre-images  of $s_{-2}$ that are in
$A_1$ until $s_{-n_{0}}\subset [\frac{3}{2\lambda}-\varepsilon ,\frac{5}{2\lambda}+\varepsilon ]\times [0,h(0)) $.
Then we take $s_{{-(n_{0}+1)}}$ as the pre-image of $s_{-n_{0}}$ that is in $A_0$.
Proceeding inductively we obtain the sequence $\{s_n\}$.$\B$

Since $s_n \subset [ -\frac{1}{2}-\varepsilon, \frac{1}{2} +\varepsilon]\times [0,g^{-1}(h(0))]\cap (A_0\cup A_1)$,
by item (1) in the definition of $h$ and $g$ we obtain that $diam (s_n)\geq \lambda ^{-n}diam (s)$.
Then there exists $n_0$ such that $s_{n_{0}}\cap W^{u}(p_0)\neq\emptyset$, 
giving a contradiction. This shows that  $W^{u}(p_0)$ is dense in
$[ -\frac{1}{2}-\varepsilon , \frac{1}{2}+\varepsilon ]\times [0,h(0)]$.

By the expansion in the first coordinate, we conclude that  $W^{u}(p_0)$ is dense in $[-1,1]\times [0,h(0)]$.
Also, as $y_0$ is an attracting point for $h$ we have that $\cup f_0^{n}([-1,1]\times [0,h(0)])$ is dense in
$[-1,1]\times [0,y_0]$. Then  $W^{u}(p_0)$ is dense in $[-1,1]\times [0,y_0]$.

  Analogously, as  $x_0$ is an attracting point for $g$ we have that $\cup f_0^{n}([-1,1]\times [0,y_0])$ is
  dense in $\T^{2}$, so we are done.
\end{proof}

For next lemma we mean by stable manifold of $p_0$  the union
$\bigcup_{n\geq 0} f^{-n}(W^s_{loc}(p_0)),$ where the local stable
manifold $W^s_{loc}(p_0)=\{(0,s):\; s\in [-1,1]\setminus\{x_1\}\}.$

\begin{lema}\label{lema9}
 The stable manifold for the point $p_0$ is dense in $\T^{2}$.
\end{lema}

\begin{proof}
Let $V\subset\T^{2}$ be an open set and let $\gamma$ be a curve contained in $V$ such that
$\gamma'(t)\subset \mathcal{C}^u_a(\gamma(t))$. By the expansion in the first coordinate of
$f_0$ there exists $n_0$ such that  $f_0^{n_{0}}(\gamma)$ 
 goes
through $A_0,$ it means that $f_0^{n_{0}}(\gamma)$ intersect both sides of $A_0$.
If $p_1\notin f_0^{n_{0}}(\gamma)$ then
$f_0^{n_{0}}(\gamma)\cap W^{s}_{loc}(p_0)\neq\emptyset$ and we are done.
 If $p_1\in f_0^{n_{0}}(\gamma)$ then there exists $n_1>n_0$  large enough such that
 $f_0^{n_{1}}(\gamma)\cap W^{s}_{loc}(q_0)\neq\emptyset$, so there exists $n_2>n_1$ such
 that $f_0^{n_{2}}(\gamma)\cap W^{s}_{loc}(p_0)\neq\emptyset$.
\end{proof}

Note that if the stable manifold and unstable manifold of $p_0$ are dense in $\T^{2}$,
by a standard proceeding we conclude that given open sets $U$ and $V$ in $\T^{2}$ there exists
$n_0\in\N$ such that  $f_0^{n_{0}}(U)\cap V\neq\emptyset$, therefore $f_0$ is transitive.

\begin{rk}
  Note that  $f_0$ is transitive and it has a saddle fix point and a repelling fix point.
  Thus $f_0$ is not hyperbolic.
\end{rk}

\begin{proof}[{\textbf{Proof of Proposition \ref{prop11}}}.]
Using $f_0$ we  construct the function $f$ that satisfies our main result. 
As $f_0$ coincides with  $A$ in the complement of a neighborhood of $A_0\cup A_1$,
we pick a point  $z_0\in\T^{2}\setminus A_0\cup A_1$ and $r>0$ such that
$f_{0}{_{|_{B(z_0,r)}}}=A_{|_{B(z_0,r)}}$. Let $z_0=(x_0,y_0)$ and consider as in the expanding case (section \ref{expC})
$\psi :\R\to\R$ such that $\psi$ is $C^{\infty}$,  $x=x_0$  the unique critical point,
$\psi (x_0)=1+\theta$ and $\psi (x)=0$ for $x$ in the complement of the intervals
$(x_0-\theta ,x_0+\theta  )$. Let $\varphi:\R\to\R$ be such that $\varphi (y)=0$ for
$y\notin [y_0,y_0+\delta ]$ and $\varphi' $ as in the expanding case but with
$\varphi'(y_0+\delta /2 )=1 $ and $\delta <2\theta <r$.  Then let us define $f:\T^{2}\to \T^{2}$  by
$$f(x,y) =\left\{
\begin{array}{l}
(\lambda x, y-\psi (x)\varphi (y)), \mbox{ if } (x,y)\in B((x_0,y_0),r)\\
{}\\
f_0(x,y), \mbox{ if } (x,y)\in B^{c}((x_0,y_0),r).
\end{array}
\right.$$
Note that the critical points of $f$ are  persistent (see Remark \ref{rk1}(d)).
The unstable manifold of the point $p_0$ is dense in $\T^{2}$, because the proof of
Lemma \ref{lema8} holds for $f$ as well.   Also note that  Lemma  \ref{conosinestables} holds for $f$,
it means that $f$ exhibits a family of unstable cones.
Therefore the proof given in Lemma \ref{lema9} also holds for  $f$. In consequence,  $f$ is transitive.
Finally we highlight that the last properties are robust. Thus Proposition \ref{prop11} follows.
\end{proof}


%
%
%
%
%
%

\begin{rk}
For $\lambda\!\leq\! 5$  define  $A_0\!=\![-\frac{1}{2\lambda^{3}}\!-\varepsilon , \frac{1}{2\lambda^{3}}\!+\varepsilon]\! \times\! S^{1}$ and
$A_1\!=\![\frac{3}{2\lambda^{3}}-\varepsilon, \frac{5}{2\lambda^{3}}+\varepsilon]\!\times\! S^{1}$,  take $h $
and $g $ such that $h^3$ and $g^3 $ are as in Figure \ref{figura6}.
Consider $f$ as in the previous case. Therefore $ f^ 3$ verifies  Proposition \ref{prop11},
so $f$ also verifies the same proposition.
\end{rk}

\subsection{Saddle case} \label{sc}
 In this section we study the saddle case, that is 
 $A$ is a $2$ by $2$ matrix such that
 $\sigma (A)=\{\lambda , \mu \}$, $|\lambda |>1>|\mu |$ and $|det(A)|\geq 2$.
Note that in this case $\lambda$ and $\mu$ are irrational numbers.
Up to change of coordinates, we may assume that  $A$ is $\left(\begin{array}{cc}
  \lambda & 0  \\
  0 & \mu  \\
\end{array}%
\right)$. Take $w=(1,0)$ as the irrational direction associated to the expanding eigenvalue
 and $w_1=(0,1)$ associated to the contracting eigenvalue.

We start this section with a lemma intrinsic to the manifold  $\T^{2}$ (Lemma \ref{conos}).
This lemma is very helpful in order to prove Theorem \ref{principal} in the saddle case.
Let us introduce first some useful notation.
We say that $w\in\R^{2}$, $w\neq 0$, is an irrational direction if
$\{\lambda w :\lambda \in\R \}\cap (\Z^{2}\setminus (0,0))=\emptyset     $ where
$\Z^{2}=\{(z_1,z_2)\in \R^{2}: \ \ z_1,z_2\in\Z \}$. For each $w,$ let $w^{\bot}\in \R^{2}$
be such that  $\{w,w^{\bot} \}$ is an orthogonal basis of
$\R^{2}$. For each  $a>0$ and $p\in\T^{2}$ we consider
$$\mathcal{C}_{a,w}^u(p)  =\left\{v=v_1w+v_2w^{\bot}\in T_p\T^{2} : \ \ \left| \frac{v_2}{v_1}    \right | <a    \right\}.  $$
Let $w_1\in\R^{2}$ be such that  $\{w,w_1 \}$ are linearly independent. Consider again for each  $a>0$ and $p\in\T^{2}$
 $$\mathcal{C}_{a,w_{1}}^s(p)  =\left\{v=v_1w_1+v_2w_1^{\bot}\in T_p\T^{2} :  \left| \frac{v_2}{v_1}    \right | <a    \right\}.  $$
Fix  $a_0>0$ small enough such that  $\mathcal{C}_{a_{0},w}^u(p) \cap  \mathcal{C}_{a_{0},w_{1}}^s(p)=\emptyset .$
Finally, given a curve $\gamma :[\alpha ,\beta ]\to \T^{2}$ define
as a length of $\gamma$ by $\ell(\gamma )=\int_{\alpha}^{\beta}| \dot{\gamma (t)}|dt$.

\begin{lema}\label{conos}
  Let  $a_0, w,w_1$,  $\mathcal{C}_{a_{0},w}^u(p)$  and $\mathcal{C}_{a_{0},w_{1}}^s(p)$ as above.
  Given $\varepsilon >0$ there exist $a_1\in (0,a_0)$ and $M>0$ such that if $\gamma $ is a curve with
  $\gamma^{'}(t)\subset \mathcal{C}^u_{a_{1},w}(\gamma(t))$  and $\ell(\gamma )\geq M$ and $\beta$
  is a curve with $\beta^{'}(t)\subset \mathcal{C}^s_{a_{1},w_{1}}(\beta(t))$ and $\ell(\beta )\geq \varepsilon$,
  then $\gamma \cap\beta\neq \emptyset .$
\end{lema}

\begin{proof}
 Since the proof is simple we leave it as an exercise for the reader.
\end{proof}

\begin{rk}
  Note that if  $a$ is such that $0<a<a_1$ as $\mathcal{C}_{a_{},w}^u(p) \subset   \mathcal{C}_{a_1,w}^u(p) $
   and $\mathcal{C}_{a_{},w_{1}}^s(p) \subset \mathcal{C}_{a_1,w_{1}}^s(p) $, the thesis of Lemma \ref{conos} holds for $a$.
\end{rk}

%
%

Let us keep in mind the main result we will prove in this section.

\begin{prop}\label{prop111}
Given a matrix  $A\in \mathcal{M}^{*}_{2\times 2}(\Z)$  with eigenvalues
$|\lambda |>1>|\mu |,$ there exist $f$ homotopic  to $A$ and $\mathcal{U}_f$
$C^{1}$-neighborhood $\mathcal{U}_f$ of $f$ such that for 
all $g\!\in\!\mathcal{U}_f$ holds that $g$ is transitive, admits a family of unstable cones  and $S_g$ is non-empty.
  \end{prop}

\subsubsection{\bf{Sketch of the Proof}}
Let $A$ as in Proposition \ref{prop111}. We perturb $A$ in a similar way as we did for
the expanding case obtaining a map such that admits unstable and stable cone fields (Lemma \ref{conos-inestablesII}
and Lemma \ref{conos-estables}) in a robust way (Lemma \ref{clly12} and Lemma \ref{clly1}). 
Given $U$ and $V$ two open sets, pick $\gamma$ a curve in $U$ that belong to the unstable cone and
$\beta$ in $V$ belonging to the stable cone, iterate forward $\gamma$ and iterate backward $\beta$
until reaching certain fix diameter, the these curves intersect (Lemma \ref{conos} and Lemma \ref{preorbita}).
Since all these process are robust we show that the perturbation is robust transitive, admits unstable cones field and
exhibits non-empty persistent critical set.

\subsubsection{\bf{Construction of $f$}}
The function $f$ that we  construct for proving our main result is inspired in  Example  2.9 in \cite{br}.
Choose $(x_0,y_0), (x_1,y_1)\in\T^{2}$  such that $(x_1,y_1)\neq (x_0,y_0),$ $A(x_0,y_0)\neq (x_0,y_0)$ and
  $A(x_0,y_0)=A (x_1,y_1)$. Fix $r>0$  such that
 \begin{itemize}
\item $\overline{A(B((x_0,y_0),3r))}\cap \overline{B((x_0,y_0),3r)}=\emptyset$ and
\item $\overline{B((x_1,y_1),3r)}\cap \overline{B((x_0,y_0),3r)}=\emptyset$.
\end{itemize}

Consider  as in the expanding case, section \ref{expC},
$\psi :\R\to\R$ such that $\psi$ is $C^{\infty}$, $x_0$ as the unique critical point,
$\psi (x_0)=1+\theta$ and $\psi (x)=0$ for $x$ in the complement of the intervals
$(x_0-\theta ,x_0+\theta  )$. Let $\varphi:\R\to\R$ be such that $\varphi (y)=0$ for
$y\notin [y_0,y_0+\delta ]$ and $\varphi' $ as in the expanding case but with
$\varphi'(y_0+\delta /2 )=1 $ and $\delta <2\theta <r$.  Then define $f:\T^{2}\to \T^{2}$  by
$$f(x,y)=f_{\delta,\theta}(x,y)=(\lambda x,\mu y-\psi (x)\varphi (y)).$$
 We denote $f_{\delta,\theta}$ by $f$ by abuse of notation.

  \begin{rk}\label{rk44}
   Note that:
   \begin{enumerate}
    \item $f|_{\T^{2}\setminus B((x_0,y_0),r)} =A|_{\T^{2}\setminus B((x_0,y_0),r)},$
    \item There exists a $C^{1}$-neighborhood $\mathcal{U}_f$ of $f$ such that
          $S_g\neq\emptyset$  for all $g\in \mathcal{U}_f$ (see Remark \ref{rk1}[item (d)]) and
    \item $f$ is homotopic to $A$.
\end{enumerate}
\end{rk}

The next lemma shows that $f$ as we defined above exhibits a family
of unstable cones.

\begin{lema}[Existence of unstable cones for $f$]\label{conos-inestablesII}
Given $\theta >0$, $a>0$, $\delta >0$  and $\lambda^{'}$ with $1< \lambda^{'}<|\lambda|$,
there exist $a_0>0$, and $\delta_0 >0$ with $0<a_0<a$ and $0<\delta_0 <\delta$ such that if
$f=f_{\theta ,\delta_{0} }$   then the following properties hold:
\begin{enumerate}
  \item[(i)] If $\mathcal{C}_{a_{0}}^u(p)=\{(v_1,v_2) : \  |v_2|/|v_1| <a_0   \}$ then $\overline{Df_p(\mathcal{C}_{a_{0}}^u(p))}\setminus \{(0,0)\}\subset \mathcal{C}_{{a_{0}}}^u(f(p))$, for all  $p\in \T^{2}$,
  \item[(ii)] if $v\in \mathcal{C}_{a_{0}}^u(p)$ then $|Df_p(v)|\geq \lambda'|v|$ and
\item[(iii)] if $\gamma$ is a curve such that $\gamma^{'}(t)\subset \mathcal{C}^u_{a_{0}}(\gamma(t))$ then $diam(f(\gamma))\geq \lambda^{'} diam (\gamma )$.
\end{enumerate}
\end{lema}

\begin{proof}
We omit the proof, since it is similar to the one given for Lemma \ref{conosinestables}.
\end{proof}

The properties given in Lemma \ref{conos-inestablesII} are open, thus it follows the next
result.

\begin{lema}\label{clly12}
For each $f$ in the hypotheses of previous lemma there exists a  $C^{1}$-
neighborhood $\mathcal{U}_f$ of $f$ such that for every $g\in
\mathcal{U}_f$  the properties $(i)$, $(ii)$ and $(iii)$ of Lemma \ref{conos-inestablesII} hold.
\end{lema}

Let $f=f_{\delta , \theta}$ be as in the hypotheses of the previous lemmas.
Let $p,q\in \T^{2}\setminus B((x_0,y_0),r)$ be such that $f(p)=q$.
Let $V$ be a neighborhood of $p$ such that $f|_{V}:V\to f(V)$ is a
diffeomorphism and $\phi:f(V)\to V$ is a local inverse of $f$, that is $\phi\circ f=id_{|_{V}}$.
As $f_{|_{\T^{2}\setminus B((x_0,y_0),r)}}$ coincide with the matrix $A$ we have the following.

\begin{lema}[Existence of stable cones for $f$]\label{conos-estables}
   Given $a>0$, $\mu_1\in (|\mu|,1),$ there exists
     $a_0\in (0,a)$ such that for every $p,q\in \T^{2}\setminus B((x_0,y_0),r)$
     with $f(p)=q$ and $\phi$ a local inverse holds the following properties:
   \begin{enumerate}
   \item[(i)] If $\mathcal{C}_{a_0}^s(q)=\{(v_1,v_2) : \  |v_1|/|v_2| <a_0 \},$
             then $\overline{D\phi_q(\mathcal{C}_{a_0}^s(q))}\subset \mathcal{C}_{a_0}^s(p)$;
   \item[(ii)] If $v\in \mathcal{C}_{a_0}^s(q),$ then $|D\phi_q(v)|\geq \mu_1 ^{-1}|v|;$ and
   \item[(iii)] If $\gamma$ is a curve such that $\gamma^{'}(t)\subset \mathcal{C}^s_{a_0}(\gamma(t)),$
             then $diam(\phi(\gamma))\geq \mu_1^{-1} diam (\gamma )$.
\end{enumerate}
\end{lema}

\begin{rk}
Note that for every $a'\in (0,a_0)$ the properties (i), (ii) and (iii) of Lemma \ref{conos-estables} hold.
\end{rk}

Once again using that properties in Lemma \ref{conos-estables} are open,
we have the following.

\begin{lema}\label{clly1}
For each  $f$ in the hypotheses of the previous lemma, there exists a $C^{1}$-
neighborhood $\mathcal{U}_f$ of $f$ such that for every $g\in
\mathcal{U}_f$  the properties  $(i),$  $(ii)$ and $(iii)$ of Lemma \ref{conos-estables} hold.
\end{lema}

Given $p\in \T^{2}$, we say that $\{p_n\}_{n\in\Z^{-}}$ is an inverse branch of $p$ if $f(p_{n})=p_{n+1}$
with $n\in\Z^{-}$ and $p_0=p$.

\begin{lema}\label{preorbita}
 There exists $\mathcal{U}_f$  $C^{1}$-neighborhood  of $f$ such that for every
 $g\in\mathcal{U}_f$ and for every $p\in \T^{2}$ there exists an inverse branch
 $\{p_n\}_{n\in\Z^-}$, by $g$,  of $p$ such that $p_n\in \T^{2}\setminus B((x_0,y_0),3r)$ for all $n< 0$.
\end{lema}

\begin{proof} From the construction of $f$ there exists a $C^{1}$-neighborhood $\mathcal{U}_f$ of $f$
such that for all $g\in\mathcal{U}_f$ and for all $p\in \T^{2}$ there exists
$p_{-1}\in \T^{2}\setminus B((x_0,y_0),3r)$ such that $g(p_{-1})=p.$
Given $p\in\T^{2}$, taking $p_0=p$, there exists $p_{-1}\in \T^{2}\setminus B((x_0,y_0),3r)$ such that $g(p_{-1})=p_0.$
Proceeding inductively we conclude the lemma.
\end{proof}

Let $w$ be the irrational direction given by the eigenvector of $A$  associated to the eigenvalue of absolute value greater than one. Note that as we are considering the matrix  $A$ as a diagonal matrix then $w=(1,0)$. For $\varepsilon =r$ and $w_1$ the stable direction of $A$, consider  $a_1>0$ and $M>0$
 given by Lemma \ref{conos}. For $a_1, \delta, \theta,$ with $\delta<2\theta<r,$ there exist  $a_0<a_1$, $\delta_0<\delta$ and $\theta_0<\theta$ such that
  Lemma \ref{conos}, \ref{conos-inestablesII} and \ref{conos-estables} are satisfied.
So considering $f=f_{\delta_0,\theta_0 }$ and $a_0$ as before,
 by Lemma \ref{clly1} there exists a $C^{1}$-
neighborhood $\mathcal{U}_f$ of $f$ such that for every $g\in
\mathcal{U}_f$  the properties  $(i),$  $(ii)$ and $(iii)$  of Lemma \ref{conos-estables} hold.

\begin{proof}[\textbf{Proof of Proposition \ref{prop111}}]
Let $\mathcal{U}_f$ be as above and $g\in \mathcal{U}_f$. Given $V_1,V_2$ open sets of $\T^{2}$ we will prove that there exists $n_0\in\N$ such that  $g^{n_{0}}(V_1)\cap V_2\neq\emptyset$. Let $\gamma $ be a curve, $\gamma\subset V_1$ such that
$\gamma^{'}(t)\subset \mathcal{C}^u_{a_0}(\gamma(t))$ and $n_1\in\N$ such that
$\ell(g^{n_{1}}(\gamma (t)))>M$. On the other hand, let  $p\in V_2$ and
$q\in \T^{2}\setminus B((x_0,y_0),3r)$ be with $g(q)=p$ and $W$ a neighborhood of $q$ such that
$g(W)\subset V_2$ and $W\subset  \T^{2}\setminus B((x_0,y_0),3r)$.
By Lemma \ref{preorbita} there exists an inverse branch $\{q_n\}_{n\in\Z^{-}}$  of $q$ such that
$q_n\in \T^{2}\setminus B((x_0,y_0),3r)$ for all $n< 0$.

 We construct a finite family of curves $\beta_0,\ldots,\beta_m$ such that $\beta_0(0)=q,\ldots,\beta_m(0)=q_m$
 with $\beta'_i(t)\subset \mathcal{C}^s_{a_0}(\beta_i(t))$ for $i=0,\ldots,m$ and $\ell(\beta_m)>r$.
  Let $\beta_0 :[-1,1]\to W$ be with $\beta_0(0)=q$ and $\beta'_0(t)\subset \mathcal{C}^s_{a_0}(\beta_0(t))$.
Let $\beta_1$ the longest curve with $\beta_1(0)=q_{-1}$  such that $g(\beta_1)\subseteq \beta_0.$
If
$\beta_1\cap B((x_0,y_0),r)=\emptyset$, by Lemma \ref{conos-estables} we have that  $\ell(\beta_1)>\mu_1 \ell(\beta_0)$.
Otherwise, 
$\beta_1 \cap B((x_0,y_0),r)\neq\emptyset$.
Since $dist(q_{-1},B((x_0,y_0),r))>2r,$ we obtain that
 $\ell(\beta_1)>r$. Proceeding inductively we obtain the family we requiered.

Now, by Lemma \ref{conos} we have that $g^{n_1}(\gamma)\cap\beta_m\neq\emptyset$.
As $\gamma\subset V_1$ and $\beta_m\subset g^{-m-1}(V_2)$, $g^{n_1+m+1}(V_1)\cap V_2\neq\emptyset$.
Taking $n_0=n_1+m+1$ we finish the proof.

\end{proof}

\section{Proof of Theorem \ref{teoA}}\label{sec2}

In this section we prove that for every robustly transitive map with persistent critical set,
the dimension of the kernel of the differential is less or equal $1$. Moreover,
there exists a sufficiently close transitive map with non-empty interior of the critical set and exhibiting 
a residual set of critical points with dense forward orbit.


Let $f\in C^{r}(M),$ 
we say that
$(x_0,y_0)\in S_{f}$ is a critical point of \textbf{fold type} if there exist neighborhoods $U$
and $V,$ of $(x_0,y_0)$ and $f(x_0,y_0)$ respectively, and local diffeomorphisms
$\psi_1:\mathbb{R}^{2}\to U$ and $\psi_2:V\to \mathbb{R}^{2}$ such that
$\psi_2\circ f\circ \psi_1 (x,y)=(x,y^{2})$, for every $(x,y)\in U$.
A point $(x_0,y_0)\in S_{f}$ is a critical point of  \textbf{cusp type} if
there exist neighborhoods and local diffeomorphisms as above such that
 $\psi_2\circ f\circ \psi_1(x,y)=(x,-xy+y^{3}) $ for every $(x,y)\in U$.

For proving Theorem \ref{teoA} we need to invoke a very classical result of singularities theory.
Let us state the Theorem of Whitney,
which classify the critical points of a generic set of endomorphisms
of class $C^3$ in any manifold of dimension two.

\begin{thm}[Whitney, \cite{w}] \label{teo1}
There exists an open and dense set ${\mathcal G}(M)$ of $C^{r}(M)$ $(r\geq 3)$
such that for every $f\in {\mathcal G }(M)$
holds that:
\begin{enumerate}
\item[i)] $S_f$ is either a submanifold of dimension $1$ or an empty set.
\item[ii)] Every critical point of $f$ is either of fold or cusp type.
\item[iii)] The cusp critical points are isolated and continuous with respect to $f$.
\end{enumerate}
\end{thm}



Now we are able to prove Proposition \ref{prop-teoA}.

\begin{proof}[\textbf{Proof of Proposition \ref{prop-teoA}}]
Given $f\in C^1(M)$ robustly transitive with persistent
critical points  and a $C^1-$neighborhood $\mathcal{U}_{f}$ of $f,$ since the maps of class $C^3$ are $C^1-$dense in the set of maps
of class $C^1,$ then there exists $g\in \mathcal{U}_{f}$ of class $C^3$ satisfying the conditions of Theorem \ref{teo1}.
Without loss of generality we may assume that $(0,0)$ is a critical point of fold type, hence there exists a neighborhood
$U$ of $(0,0)$ such that $g(x,y)= (x, y^2)$ for every $(x,y)\in U.$

Let us prove the theorem for $g$. That is, there exists $\bar{g}$ $C^1-$close to $g$ such as the interior of the critical set is non-empty.
Given $\varepsilon>0,$ choose $\delta>0$ such that $4\delta<\varepsilon$ and $\delta^2<\varepsilon.$
Consider $\varphi:\mathbb{R}\to \mathbb{R}$ a bump function of class $C^1$ such as in Figure \ref{figura3} with
$|\varphi'|\leq \frac{2}{\delta}$.
Consider $\bar{g}$ a $C^1$-perturbation of $g$ defined by $\bar{g}(x,y)=(x,\varphi(y)y^2)$, for $(x,y)\in U.$
Since it is not hard to prove that the distance between $\bar{g}$ and $g$ is less than $\varepsilon$
in the $C^1$ topology, we leave the details for the reader.
Moreover, the critical set of $\bar{g}$ contain the ball centered  at $(0,0)$ and radius $\frac{\delta}{2}.$
 \end{proof}

\begin{figure}[ht]
\psfrag{d}{\tiny{$\delta$}}\psfrag{d2}{\tiny{$\delta/2$}}
\psfrag{-d2}{\tiny{$-\delta /2$}}
\psfrag{-d}{\tiny{$-\delta$}}
\psfrag{1}{\tiny{$1$}}
\psfrag{aa}{\tiny{$\theta$}}
\psfrag{q}{\tiny{$\psi$}}
\psfrag{phi}{\tiny{$\varphi$}}
\begin{center}
\includegraphics[scale=0.15]{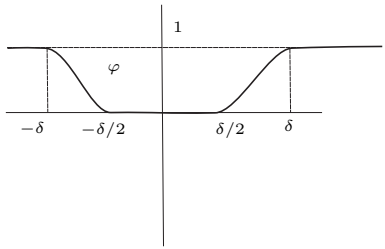}
\caption{Graph of $\varphi$}\label{figura3}
\end{center}
\end{figure}

Using Theorem \ref{teo1} and Proposition \ref{prop-teoA} follows the proof of Theorem \ref{teoA}.

\begin{proof} [\textbf{Proof of Theorem \ref{teoA}}]

Proof of $(1).$ Given $f$ as in the statement, let $g$ be as in item $(1)$. Since
$g$ is transitive, there exists a residual set of points with dense orbit for $g$.
Moreover, by Proposition \ref{prop-teoA} the interior of the critical set $S_g$ of $g$ is non-empty, then
there exists a residual set in $S_g$ with dense orbit.

Proof of $(2).$ Finally we prove that the kernel of the differential of robustly transitive maps has dimension less
or equal $1$. In fact, let us assume by contradiction that there exists $f(x,y)=(f_1(x,y),f_2(x,y))$ robustly transitive map such as
$$\frac{\partial f_i}{\partial x}(0,0)=\frac{\partial f_i}{\partial y}(0,0)=0,\; i=1,2.$$
Given $\varepsilon>0,$ choose $\delta>0$ such that
\begin{itemize}
\item $f(B((0,0),\delta))\subseteq B(f(0,0),\varepsilon/5);$
\item If $(x,y)\in B((0,0),\delta),$ then
        $\frac{\partial f_i}{\partial x}(x,y)< \varepsilon/5,\;
        \frac{\partial f_i}{\partial y}(x,y)<\varepsilon/5,$ for $i=1,2.$
\end{itemize}

Let $\varphi:\mathbb{R}\to\mathbb{R}$ such as in the Figure \ref{figura3} and $|\varphi'|<2/\delta.$
Define $g(x,y)=f(0,0)+\varphi(x^2+y^2)(f(x,y)-f(0,0)).$ Note that $g(B((0,0),\delta/2))=f(0,0).$ Thus, $g$
is not transitive obtaining a contradiction. We leave the details of the calculation for the reader.
\end{proof}

\subsection*{Acknowledgments:}
The authors are grateful to E. Pujals and A. Rovella for
useful and encouraging conversations and suggestions. The authors
are also grateful for the nice environment provided by IMPA, DMAT (PUC-Rio),
Universidad de Los Andes(Venezuela) and Universidad de la Rep\'ublica(Uruguay)
during the preparation of this paper. We are grateful for the valuable corrections
and careful reading of the anonymous referee.  

%


\end{document}